\documentclass{amsart}
\usepackage{amscd}
\usepackage{amssymb}
\usepackage{amsmath}
\usepackage{amsthm,hyperref}

\usepackage{color}

\title{Random triangular Burnside groups}
\author{Dominik Gruber}
\email{dominik.gruber@math.ethz.ch}
\address{Department of Mathematics \\ ETH Zurich \\ Zurich, CH}
\author{John M. Mackay}\thanks{The authors would like to thank the Isaac Newton Institute for Mathematical Sciences, Cambridge for support and hospitality during the programme ``Non-positive curvature, group actions and cohomology'' where work on this paper was undertaken, supported by EPSRC grant EP/K032208/1.  The research of the second author was also supported in part by EPSRC grant EP/P010245/1.}
\email{john.mackay@bristol.ac.uk}
\address{School of Mathematics \\ University of Bristol \\ Bristol, UK}
\date{\today}

\usepackage{amsmath,amsthm,amsfonts,graphicx,amssymb}
\usepackage{hyperref}
\AtBeginDocument{%
   \def\MR#1{}
}

\numberwithin{equation}{section}
\newtheorem{theorem}[equation]{Theorem}
\newtheorem{proposition}[equation]{Proposition}
\newtheorem{corollary}[equation]{Corollary}
\newtheorem{lemma}[equation]{Lemma}

\newtheorem{notation}[equation]{Notation}

\newtheorem{question}[equation]{Question}

\newtheorem{definition}[equation]{Definition}
\newtheorem{remark}[equation]{Remark}

\newtheoremstyle{citing}
  {3pt}
  {3pt}
  {\itshape}
  {}
  {\bfseries}
  {}
  {.5em}
  {\thmnote{#3}}

\theoremstyle{citing}

\newcommand{\cP}{\mathcal{P}}

\newcommand{\ra}{\rightarrow}
\newcommand{\R}{\mathbb{R}}

\newcommand{\N}{\mathbb{N}}
\newcommand{\Z}{\mathbb{Z}}

\DeclareMathOperator{\Red}{Red}
\DeclareMathOperator{\Cancel}{Cancel}

\newcommand{\bestboundx}{\frac{11}{12}-\frac{\sqrt{41}}{12} \approx 0.38307}
\newcommand{\bestboundshort}{\frac{11}{12}-\frac{\sqrt{41}}{12}}

\begin{document}

\begin{abstract}
We introduce a model for random groups in varieties of $n$-periodic groups as $n$-periodic quotients of triangular random groups. We show that for an explicit $d_{\mathrm{crit}}\in(1/3,1/2)$, for densities $d\in(1/3,d_{\mathrm{crit}})$ and for $n$ large enough, the model produces \emph{infinite} $n$-periodic groups. As an application, we obtain, for every fixed large enough $n$, for every $p\in (1,\infty)$ an infinite $n$-periodic group with fixed points for all isometric actions on $L^p$-spaces.
Our main contribution is to show that certain random triangular groups are uniformly acylindrically hyperbolic.
\end{abstract}

\maketitle

\section{Introduction}\label{sec-intro}

  From the viewpoint of combinatorial group theory, random groups are a natural object to study: what are the typical properties of quotients of free groups by normal subgroups generated by randomly chosen elements?  The study of such questions took off in the late 1980s, beginning with pioneering work of Gromov~\cite{Gro-87-hyp-groups, Gro-91-asymp-inv, Gro-00-spaces-and-qs, Gro-03-rw-rg}, and vastly developed since by Arzhantseva, Champetier, Ol'shanskii, Ollivier and many others; for a survey see Ollivier~\cite{Oll-05-rand-grp-survey}.  
  
  It is natural to consider different families of random groups by replacing the free (non-abelian) groups by some other groups.  Using a free abelian group leads to random abelian groups, studied by Dunfield--Thurston~\cite{Dun-Thu-06-Random-3-mflds} and others; using a free nilpotent group leads to random nilpotent groups studied by Cordes--Duchin--Duong--Ho--S\'anchez~\cite{CDDHS-18-Random-nilpotent}.
  
  In this paper we introduce a model for random Burnside groups, by taking suitable quotients of free Burnside groups.
  Recall that the free Burnside group $B(m,n)$ of rank $m$ and exponent $n$ is given by $B(m,n)\cong F_m/F_m^n$, where $F_m$ is the free group of rank $m$ and, for a group $G$, we let $G^n:=\langle g^n:g\in G\rangle\leq G$, i.e.\ the verbal subgroup corresponding to $w^n$. 
  Free Burnside groups themselves are mysterious: finite for $n=2,3,4,6$ \cite{Burnside,Sanov,Hall} but, by major work of Novikov, Adian, Ivanov, and Lysenok known to be infinite for large $n$ \cite{Novikov-Adjan-1,Novikov-Adjan-2,Novikov-Adjan-3,Ivanov,Lysenok}.
  
  We define a triangular model for quotients of free Burnside groups as an extension of the usual triangular model for random groups, which are defined using a `density' parameter $d \in (0,1)$ determining how many random relations of length $3$ are chosen.  Our goal is to show that for a suitable range of $d$, typical random triangular Burnside groups are both infinite and not free Burnside groups.
  In order to do so, we show the independently interesting result that random triangular groups are uniformly acylindrically hyperbolic, as we discuss further below.

Formally, a generating set $S$ of a group $G$ is an epimorphism $F(S)\to G$, where $F(S)$ is the free group on $S$.
\begin{definition}\label{def:model}
Let $\mathcal G:=(G_m)$ be a sequence of groups with generating sets $S_m$ of size $m$. Let $d \in (0,1)$ be chosen, called the \emph{density}. Let $\cP$ be a property of a group. For each $m \in \N$, consider the probability distribution on the set of all tuples $R$ of elements of $F(S_m)$ found by choosing $\lfloor(2m-1)^{3d}\rfloor$ cyclically reduced words of length $3$ uniformly and independently at random. We say that a \emph{random triangular quotient of $\mathcal G$ at density $d$ has property $\cP$ asymptotically almost surely (a.a.s.)} if the probability that the quotient of $G_m$ by $R$ has $\cP$ goes to $1$ as $m \ra \infty$.

We call a random triangular quotient of the sequence of free groups of rank $m$ with canonical generating sets a \emph{random triangular group}. Given $n>0$, we call a random triangular quotient of  the sequence of free Burnside groups of rank $m$ and exponent $n$ with canonical generating sets a \emph{random triangular $n$-periodic group}.
\end{definition}
Some references instead choose $R$ as a random tuple of $\lfloor (2m-1)^{3d}\rfloor$ \emph{distinct} words; this variation leads to equivalent models at densities $d<\frac12$ \cite[I.2.c]{Oll-05-rand-grp-survey}.
Our main result is the following: 
\begin{theorem}\label{thm:main}
	For each $d_0<\bestboundx$ there exists $n_0$ so that for each $n\geq n_0$ and for each $0 < d \leq d_0$, a random triangular $n$-periodic group $G$ at density $d$ is infinite.
\end{theorem}

At densities $d<\frac13$, random triangular groups are free (see \cite[Theorem 1]{ALS-15-random-triangular-at-third} and discussion), thus random triangular $n$-periodic groups at such densities are just free Burnside groups (of a lower rank), so Theorem~\ref{thm:main} is only interesting at densities $\geq \frac{1}{3}$.

At densities $> \frac{1}{3}$, random triangular groups have interesting fixed point properties: they have Kazhdan's property (T) a.a.s.~\cite{Zuk-03-prop-T,KK-11-zuk-revisited}, and indeed for each $p \in (1,\infty)$ they have property $FL^p$, that is any affine isometric action of the group on an $L^p$-space has a global fixed point~\cite{DM-16-FLp}; see these references for further details.
While these fixed-point properties are trivial for finite groups, they are highly interesting for \emph{infinite} groups.
Clearly, whenever a random triangular group at density $d$ has $\cP$ a.a.s.\ and $\cP$ is inherited by quotients, then a random triangular $n$-periodic group at density $d$ also has $\cP$. 
 We deduce, for example:
\begin{corollary}
	At densities $d$ with $\frac{1}{3} < d < \bestboundx$, for $n$ large enough, a random triangular $n$-periodic group at density $d$ is infinite and has Kazhdan's property (T) a.a.s., and for each fixed $p_0 \in (1, \infty)$, a random triangular $n$-periodic group at density $d$ is infinite and has property $FL^p$ for all $p \in (1,p_0]$ a.a.s.
\end{corollary}

In view of the recent result that, if $n$ is large enough and not a prime, then free Burnside groups of exponent $n$ and rank at least 2 do not have property (T) \cite{Osajda-cubization}, this implies that for large enough composite $n$ a random triangular $n$-periodic group at these densities is not isomorphic to a free Burnside group a.a.s.

\medskip

The key tool for proving Theorem~\ref{thm:main} is the following result. It says that a group acting non-elementarily, purely loxodromically, and acylindrically on a $\delta$-hyperbolic space admits an infinite $n$-periodic quotient, so long as $n$ is large enough, where ``large enough'' only depends on $\delta$ and the acylindricity constants of the action. An action is purely loxodromic if every non-trivial group element acts as a loxodromic.

\begin{theorem}[{\cite[Proposition 4.1]{Coulon-Gruber} and \cite[Theorem 5.7]{Coulon-even}}]\label{thm:acylind-quotient}
	Suppose that a group $G$ acts purely loxodromically and acylindrically on a $\delta$-hyperbolic space $X$ with constants $L$ and $N$, i.e.\ for any $x,y \in X$ with $d(x,y) \geq L$ we have for 
	\[ G_{x,y,800\delta} := \{ g \in G : d(gx,x) \leq 800\delta, d(gy,y)\leq 800\delta\} \]
	that $|G_{x,y,800\delta}| \leq N$.
	Then there exists $n_0=n_0(\delta,L,N)$ so for any $n\geq n_0$ the group $G / G^n$ is infinite.
\end{theorem}

In the case that $n$ has a large enough odd divisor, the above theorem is based on \cite[Theorem~6.15]{Coulon-partial}, and our statement is a slight simplification of the statement given in \cite{Coulon-Gruber} that will be sufficient for our purposes. In the case $n$ is divided by a large enough power of $2$, it follows from \cite[Theorem 5.7]{Coulon-even}. We state the result for $800\delta$ instead of $100\delta$ as in \cite{Coulon-Gruber}, as \cite{Coulon-Gruber} and \cite{Coulon-even} use the 4-point definition of hyperbolicity, while we use the slim-triangles definition, which incurs a conversion factor of 8 \cite[Chapitre 1, Proposition 3.6]{Coornaert-Delzant-Papadopoulos}.
The geometric ideas behind behind Theorem~\ref{thm:acylind-quotient} were introduced by Delzant and Gromov~\cite{Delz-Gro-08-mesoscopic}; for a fuller history we refer to the discussion in~\cite{Coulon-even}.

\medskip

In our proof, we shall use the action of a triangular random group $G$ on its Cayley graph $X$ to show that a triangular random $n$-periodic group, which can be realized as $G/G^n$, is infinite. For suitable densities $d$, $X$ is $\delta$-hyperbolic, where $\delta$ only depends on $d$, i.e.\ is independent of the rank $m$. In order to apply Theorem~\ref{thm:acylind-quotient}, it remains to also bound $L$ and $N$ independently of $m$. 
The key challenge in this is to bound $N$ independent of the volume of balls in $G$, because as $m \ra \infty$ this will not be controlled.  The main idea of the proof is to use strong isoperimetric inequalities for random groups to show that, with at most uniformly boundedly many exceptions, geodesics $[x,y]$ and $[gx,gy]=g[x,y]$ have to collide within a uniformly bounded distance of $x$. 

\medskip

In the following we show that a triangular random group at density $d_0<\bestboundx$ acts acylindrically on its $\delta$-hyperbolic Cayley graph, with acylindricity constants only depending on $d_0$.

\begin{theorem}\label{thm:acylindrical}
 Let $d_0<\bestboundx$. Then there exist $\delta=\delta(d_0)$, $L=L(d_0)$ and $N=N(d_0)$ such that for any $d\leq d_0$, for a random triangular group $G$ at density $d$ a.a.s.\ the Cayley graph $X$ is $\delta$-hyperbolic, and for every $x,y\in X$ with $d(x,y)\geq L$ we have $|G_{x,y,800\delta}|\leq N$.
\end{theorem}
The claim on $\delta$-hyperbolicity of $X$ is well-known, see Theorem~\ref{thm:isop2}, and stated here to emphasise the dependence on $d_0$. Furthermore, a.a.s.\ $G$ is infinite, torsion-free, and not isomorphic to $\Z$ by \cite[V.d.]{Oll-05-rand-grp-survey}, whence the action on $X$ is non-elementary and purely loxodromic, and we can deduce Theorem~\ref{thm:main} from Theorem~\ref{thm:acylindrical} using Theorem~\ref{thm:acylind-quotient}.
In the remainder of the paper, we will prove Theorem~\ref{thm:acylindrical}.

\medskip

The bound $d_0 < \bestboundx$ results from our current methods, which use that for $d$ not too much larger than $\frac13$, van Kampen diagrams bounded by two geodesics have particularly nice forms.  There is no obvious reason why more elaborate arguments could not work at higher densities, so we ask:
\begin{question}
	For each $d_0 < \frac{1}{2}$, does there exist $n_0$ so that for each $n\geq n_0$ and for each $0 < d \leq d_0$, a random triangular $n$-periodic group $G$ at density $d$ is infinite?
\end{question}

\begin{remark}
 In the first version of this paper, we only considered \emph{odd} exponents $n$. This case is covered by \cite[Proposition 4.1]{Coulon-Gruber}. Since then,  \cite[Theorem 5.7]{Coulon-even} appeared, which together with \cite[Proposition 4.1]{Coulon-Gruber} covers the case of even exponents $n$. Therefore, Theorem~\ref{thm:acylind-quotient} holds without restrictions on the parity of $n$, whence so does Theorem~\ref{thm:main}.
\end{remark}

\subsection*{Acknowledgements}
We thank R\'emi Coulon and an anonymous referee for helpful correspondence.

\section{Isoperimetric inequalities}

	Ollivier established an isoperimetric inequality for random groups that is a key tool in our argument, see Theorem~\ref{thm:isop2} below.
	We are going to need a version of this inequality which applies to not-necessarily-planar $2$--complexes, which are not necessarily reduced.  (For a different generalisation to non-planar diagrams, see Odrzyg\'o\'zd\'z~\cite[Theorem 1.5]{Od-14-nonplanar}.)

	Let us recall the language of combinatorial complexes and van Kampen diagrams~\cite[I.8A]{BH-99-Metric-spaces}, \cite{McC-Wise-02-Fans-ladders}.
	A map $Y\to K$ between CW-complexes is \emph{combinatorial} if it maps each open cell of $Y$ homeomorphically onto an open cell of $X$; a CW-complex is \emph{combinatorial} if for suitable subdivisions the attaching maps of cells are all combinatorial. We shall use the expressions ``$1$--cell'' and ``edge'' interchangeably.
	
	Since our random model constructs presentations by independently sampling random words, it will be convenient to think of tuples of relators, rather than sets of relators. To that end, throughout the rest of the paper, consider a group presentation as a pair $\langle S\mid R\rangle$ where $R$ is a tuple of elements of $F(S)$. As usual, this defines the group $F(S)/N$, where $N$ is the normal closure of the set of elements of $R$.

	Having fixed a group presentation $\langle S \mid R \rangle$ for $G$, we let $K=K(G;S,R)$ be the standard combinatorial $2$--complex for this presentation, with one $0$--cell, a $1$--cell  for each generator, and a $2$--cell for each relation.
	A \emph{van Kampen diagram $D$} is a non-empty, contractible, finite combinatorial $2$--complex $D$ embedded in the plane together with a combinatorial map $D\to K$.
	(The data of the map $D \to K$ is equivalently given by a labelling of $1$--cells by generators so that the attaching map of each $2$--cells is labelled by a relation and, in the case there are multiple $2$--cells with the same boundary cycles, choices among these $2$--cells.)
	Let $\overline K$ be the complex obtained from $K$ by identifying $2$--cells with the same boundary cycles. The van Kampen diagram $D$ is \emph{reduced} if the composition $D \to K\to \overline K$ is an immersion around open $1$--cells.

	Let $Y$ be a combinatorial $2$--complex, and denote by $Y^{(i)}$ the set of $i$--cells of $Y$.
	Let $|Y| = |Y^{(2)}|$ be the number of $2$--cells of $Y$.
 In a slight variation of \cite{MacPrz-15-balanced-walls}, the \emph{cancellation} of $Y$ is defined as
\[
	\Cancel(Y) = \sum_{e \in Y^{(1)}} (\deg(e)-1)_+,
\]
where $\deg(e)$ is the number of times $e$ appears as the image of an edge of the attaching map of a $2$--cell of $Y$, and $(\cdot)_+ = \max\{\cdot,0\}$.
If $D$ is a van Kampen diagram, then $\Cancel(D)$ is the number of internal edges of $D$, and so Ollivier's isoperimetric inequality can be phrased as:
\begin{theorem}[{\cite[Theorem 2, Corollary 3]{Oll-07-sc-rand-group}, \cite[Theorem 2.2]{MacPrz-15-balanced-walls}}]\label{thm:isop2}
	For $0<d<\frac12$ and any $\epsilon>0$ a.a.s.\ every reduced van Kampen diagram $D$ in a random triangular group $G$ at density $d$ satisfies
	\[
		\Cancel(D) \leq (d+\epsilon)|D|3.
	\]
	Equivalently, every such $D$ has $|\partial D| \geq 3(1-2d-2\epsilon)|D|$, where $|\partial D|$ is the number of boundary edges.

	Consequently, the Cayley graph of $G$ is $\delta$-hyperbolic with $\delta\leq 12/(1-2d)$.
\end{theorem}
\begin{remark}
	Ollivier's original statement was for the Gromov density model where the lengths of relations grow, but the proof works essentially verbatim in the triangular model too, compare also \cite[Lemma 7]{ALS-15-random-triangular-at-third}.
\end{remark}

We now describe a generalisation of van Kampen diagrams suitable for our purposes.
For a random triangular presentation at densities $d<1/2$ a.a.s.\ the relators are all distinct even after taking cyclic conjugates or inverses, and there are no proper powers; however, to avoid conditioning on this event later, the following definition allows for such occurances.
\begin{definition}\label{def:labelled-cplx}
	Suppose $G = \langle S \mid R\rangle$ is a group presentation, where $R$ is a tuple of relators $R=(r_i)_{i\in I}$, and $K=K(G\; S,R)$.
	A \emph{labelled $2$--complex} $Y$ is a combinatorial $2$--complex $Y$ with a combinatorial map $Y\to K$.
\end{definition}
This map $Y\to K$ encodes a lot of combinatorial data which we describe with the following notation.
\begin{notation}\label{def:labelled-cplx-extra}
	Suppose we are in the situation of Definition~\ref{def:labelled-cplx}.
	Let $\pi:K^{(2)}\to I$ be the map identifying each $f\in K^{(2)}$ with the corresponding index $i \in I$ such that each $f$ has a boundary path with label $r_{\pi(f)}$  (i.e.\ each $2$--cell remembers its position in $R$).
	
	For each $2$-cell $f \in Y^{(2)}$, let $\partial f$ be the choice of boundary path (expressed as a tuple of oriented $1$--cells) such that, if $\pi(f)=i$, then the path $\partial f$ bears the word $r_i$. Here we also denote by $\pi$ the map $Y^{(2)}\to K^{(2)}\to I$, and we consider edges of $Y$ with the orientations and labels by $S$ inherited from $K$. 
\end{notation}
We also need a way to measure how (un)reduced a labelled $2$--complex is.
\begin{definition}\label{def:red}
	Suppose $Y$ is a labelled $2$--complex.
	Suppose we have $e \in Y^{(1)}$ and $f \in Y^{(2)}$ with boundary $\partial f$ labelled $e_1,e_2,\ldots,e_n$, and that $f$ \emph{contains} $e$, i.e.\ some $e_k$ is an oriented version of $e$.
	We say the \emph{least position of $e$ in $f$} is the least positive natural number $k$ such that $e_k$ is an oriented version of $e$.

	Given $e \in Y^{(1)}$ and $f \in Y^{(2)}$
	we set $\xi(e,f)=1$ if $f$ contains $e$ and, for any $2$--cell $f'$ with $\pi(f')=\pi(f)$ that contains $e$, the least position of $e$ in $f$ is less than or equal to the least position of $e$ in $f'$.  Otherwise, set $\xi(e,f)=0$.
Let
	\[
		\Red(Y) = \sum_{e \in Y^{(1)}}
		\sum_{i\in I}\bigl(\sum_{f\in Y^{(2)},\pi(f)=i
		} \xi(e,f)-1\bigr)_+.
	\]
\end{definition}

For example, if an edge $e$ has six $2$--cells with label $i\in I$ attached to it, and in four of these $2$--cells it is attached as the second edge and in two it is the third edge, then $e$ and $i$ give a contribution of $4-1=3$ to $\Red(Y)$.  
This number $\Red(Y)$ is relevant in our probabilistic arguments later which require estimates on how a relator interacts with itself in a diagram.

Note too that any van Kampen diagram $D$ over $\langle S\mid R\rangle$ gives rise to a labelled $2$--complex. In the case $R$ contains proper powers and/or multiple copies of words that are cyclically conjugate or cyclically conjugate up to inversion, this involves choices. However, if $D$ is reduced then, for any choice, we have $\Red(D)=0$, because for each $1$--cell $e$ there exists at most two $2$--cells containing $e$, and not both can bear the same relator and have $e$ at the same position.

As mentioned before, for a random triangular presentation at densities $d<1/2$ a.a.s.\ $R$ does not contain proper powers and all words in $R$ are distinct even after taking cyclic conjugates or inverses.  In this case, a van Kampen diagram $D$ is reduced if and only if $\Red(D)=0$.  The definition of $\Red(D)$ above allows us to state and prove results avoiding conditioning on $R$ satisfying this property.

We now generalise Ollivier's Theorem~\ref{thm:isop2}.
\begin{theorem}\label{thm:isop-reductions2}
	For any $0<d<\frac12$, $M \in \N$, $\epsilon >0$, a.a.s.\ in a random triangular group $G$ at density $d$ every labelled $2$--complex $Y$ with at most $M$ $2$--cells satisfies
	\[
		\Cancel(Y)-\Red(Y) \leq (d+\epsilon)|Y|3.
	\]
\end{theorem}
In order to prove this theorem, we adapt the language of~\cite[Section V.b.]{Oll-05-rand-grp-survey}. 
\begin{definition}[cf.\ {\cite[Definition 57]{Oll-05-rand-grp-survey}}]\label{def:labelled-cplx-abstract}
	An \emph{abstract labelled $2$--complex} is a combinatorial $2$--complex $Y$ together with
	\begin{enumerate}
		\item a choice of $n \in \{1,\ldots,|Y|\}$ called the \emph{number of distinct relators in $Y$}, 
	  \item a surjective assignment $\pi: Y^{(2)} \to \{1,\ldots,n\}$, and
	  \item for each $f\in Y^{(2)}$ a choice of boundary path $\partial f$ (i.e.\ an expression as a tuple of oriented $1$--cells).
	\end{enumerate}
	Given $e \in Y^{(1)}$ and $f \in Y^{(2)}$, define $\xi(e,f)=1$ as in Definition~\ref{def:labelled-cplx} and, subsequently, define $\Red(Y)$ as in Definition~\ref{def:labelled-cplx}.
\end{definition}

Let $Y$ be an abstract labelled $2$--complex with $n$ distinct relators as above. Let $k\leq n$ and
consider a $k$-tuple $(w_1,w_2,\dots,w_k)$ of words.
For each $j \in\{1,2,\ldots,k\}$ assume that the lengths of the boundary
paths of $2$--cells $f \in \pi^{-1}(j)$ are all equal to $|w_j|$.
For each $j \in\{1,2,\ldots,k\}$, $f \in \pi^{-1}(j)$ and
$i\in\{1,2,\dots, |w_j|\}$, assign the $i$--th letter of $w_j$ to the
unique oriented edge $e$ appearing as the $i$--th edge in $\partial f$.
This produces, for each oriented edge, a (possibly empty) set of
assigned letters in $S \cup S^{-1}$. We call this assignment the partial
labelling of $Y$ by $(w_1,w_2,\dots,w_k)$. We say the partial labelling
is \emph{consistent} if, for every edge $e$, the set of assigned letters
has at most one element and if both $e$ and its reversed edge $e^{-1}$
have non-empty sets of assigned letters $\{s\}$ and $\{t\}$,
respectively, then $s=t^{-1}$.

We think of the above process as writing the words $(w_1,w_2,\dots,w_k)$
on the boundaries of their corresponding $2$--cells. If this can be done
without ambiguity, the process is consistent.

\begin{definition}\label{def:fulfil}
	Fix a (random group) presentation $\langle S \mid R \rangle$.
	Let $Y$ be an abstract labelled $2$--complex with $n \leq |Y|$ distinct relators.
	A $k$-tuple of words $(w_1,\ldots,w_k)$ in $R$, $k \leq n$ \emph{(partially) fulfills} $Y$ if the partial labelling of $Y$ by $(w_1,w_2,\dots,w_k)$ is consistent.  
\end{definition}

Note that the words $w_1,\ldots,w_n$ need not be pairwise distinct.

Let $R=(r_i)_{i\in I}$ and $K$ be as given in Definition~\ref{def:labelled-cplx}. For $n\in \N$, an \emph{injective $n$--sub-tuple of $R$} is an $n$-tuple $(w_1,\ldots,w_n)$ of elements of $R$ together with an injective indexing map $\iota:\{1,\dots,n\}\to I$ such that $w_k=r_{\iota(k)}$ for each $k$. Suppose that $Y$ is fulfilled by an injective $n$--sub-tuple $(w_1,\ldots,w_n)$ of $R$. After possibly choosing arbitrary images for edges not contained in any $2$--cells, this data gives us a combinatorial map $Y\to K$ as in Definition~\ref{def:labelled-cplx}. Observe from the definitions that, as $\iota$ is injective, we get the same value for $\Red(Y)$ on considering $Y$ as an abstract labelled $2$--complex or as the labelled $2$--complex resulting from this fulfillment.

\begin{proof}[Proof of Theorem~\ref{thm:isop-reductions2}]
 We follow the proof in \cite[Section V.b.]{Oll-05-rand-grp-survey}, with minor changes.  The key step is the following:

	\begin{lemma}\label{lem:claim}	 		Assume $Y$ is an abstract labelled $2$--complex such that $\Cancel(Y)-\Red(Y)> (d+\epsilon)|Y|3$ and $|Y|\leq M$. 
	Then the probability that for an $\lfloor(2m-1)^{3d}\rfloor$--tuple of random triangular relators there exists an injective $n$--sub-tuple of relators fulfilling $Y$ (where $n \leq |Y|$ is the number of distinct relators in $Y$) goes to $0$ as $m\to\infty$. 
\end{lemma}

	An induction argument shows that it is sufficient to consider the case every cell of $Y$ is contained in a $2$--cell. Due to the restriction that $|Y|\leq M$, there are only finitely many such abstract complexes to consider, whence the above shows that the probability that at least one of these complexes admits an $n$--tuple of relators fulfilling it goes to $0$. 
	
		The theorem then follows.  Indeed, let $K$ be the complex associated to a random triangular presentation $\langle S\mid R\rangle$ as in Definition~\ref{def:labelled-cplx}. If $Y\ra K$ is a labelled $2$--complex with $n$ different $2$--cells in the image of $Y$, then $Y$ is also the fulfilment of an abstract labelled $2$--complex by an injective $n$--sub-tuple $(w_1,\ldots,w_n)$ of $R$.  Hence, as the probability of there existing a labelled $2$--complex $Y$ with $\Cancel(Y)-\Red(Y)>(d+\epsilon)|Y|3$ and $|Y| \leq m$ is bounded above by the probability of there existing an abstract labelled $2$--complex as in the claim above, and that probability goes to $0$ as $m\to\infty$, the theorem follows.
\end{proof}
	It remains to prove Lemma~\ref{lem:claim}.
\begin{proof}[Proof of Lemma~\ref{lem:claim}]
	In the abstract labelled $2$--complex $Y$, let $m_1, \ldots, m_n$ be the number of $2$--cells labelled by each of $1,\ldots,n$; without loss of generality, we may assume that $m_1 \geq m_2 \geq \cdots \geq m_n \geq 1$, since $\Red(Y)$ is invariant under permutation of labels. Given an edge $e$ and a $2$--cell $f$, we set $\chi(e,f)=1$ if $f$ contains $e$ and the label of $f$ is minimal among all those $2$--cells containing $e$, and $\chi(e,f)=0$ otherwise. Given an edge $e$, denote by $i_e$ the minimum of all labels of $2$--cells containing $e$; recall that all edges are contained in a $2$--cell.

	Given a $2$--cell $f$, let $$\delta(f)=|f|-\sum_{e\in Y^{(1)}}\chi(e,f)\xi(e,f),$$ where $|f|$ is the edge-length of the boundary path of $f$
	\footnote{While this edge-length is always $3$ in this paper, we use this notation to clarify the role of this quantity in the proof.}.
	If we follow the process of choosing our random relators letter-by-letter starting from the first relator in such a way that the labelling of $Y$ is consistent, 
	then $\delta(f)$ is the number of edges (counted with multiplicity) in the boundary path of $f$ whose labels are forced by earlier labellings in the process. 
	Observe here that an edge $e$ that is visited more than once by the boundary path of $f$ can only be in minimal position once in that boundary path, i.e.\ we do not need to worry about its multiplicity in the sum. We have (noting that a sum with empty index set is $0$):
	
	\begin{align*}
	 &\Red(Y)+\sum_{f\in Y^{(2)}}\delta(f) \\
	 &= \sum_{f\in Y^{(2)}}|f|+ \sum_{e\in Y^{(1)}}\Bigl(\sum_{i=1}^n\bigl(\sum_{f\in Y^{(2)}, \pi(f)=i}\xi(e,f)-1\bigr)_+-\sum_{f\in Y^{(2)}}\chi(e,f)\xi(e,f)\Bigr) \\
	 &= \sum_{e\in Y^{(1)}}\deg(e)+\sum_{e\in Y^{(1)}}\Bigl(\sum_{i=1}^n\bigl(\sum_{f\in Y^{(2)}, \pi(f)=i}\xi(e,f)-1\bigr)_+-\sum_{f\in Y^{(2)},\pi(f)=i_e}\xi(e,f)\Bigr) \\
	 &\geq \sum_{e\in Y^{(1)}}\Bigr(\deg(e)+\bigl(\sum_{f\in Y^{(2)},\pi(f)=i_e}\xi(e,f)-1\bigr)_+-\sum_{f\in Y^{(2)},\pi(f)=i_e}\xi(e,f)\Bigl) \\
	 &\geq \sum_{e\in Y^{(1)}}(\deg(e)-1)_+=\Cancel(Y).
	\end{align*}

	Thus if $\delta_i$ is the maximal value of $\delta(f)$ among all $2$--cells $f$ labelled by $i$, 
	\begin{equation*}
	  3|Y|+2(\Red(Y)-\Cancel(Y)) 
	  \geq 3|Y|-2\sum_{f \in Y^{(2)}} \delta(f)
	  \geq 3|Y|-2\sum_{1\leq i\leq n} m_i \delta_i
	\end{equation*}
	(compare \cite[V.b.(2)]{Oll-05-rand-grp-survey}).
	
	For $1 \leq i \leq n$, let $p_i$ be the probability that random cyclically reduced words $(w_1,\ldots,w_i)$ partially fulfil $Y$ (and $p_0=1$).
	Then
	\begin{lemma}[{cf.\ \cite[Lemma 59]{Oll-05-rand-grp-survey}}]
	  We have $p_i / p_{i-1} \leq (2m-1)^{-\delta_i}$.
	\end{lemma}
	\begin{proof}
	  (This is a slight variation on Ollivier, as for clarity we keep track of error estimates caused by counting cyclically reduced words.)

	  Cyclically reduced words of length three either start with a repeated symbol or not, and so the total number of these is 
	  $2m\cdot 1 \cdot(2m-1) + 2m(2m-2)(2m-2) = (2m-1)^3+1$.

	  If $\delta_i >0$, then the total number of choices of $w_i$ allowed is $\leq (2m-1)^{3-\delta_i}$, thus the probability of success is $\leq (2m-1)^{3-\delta_i}/((2m-1)^3+1)\leq (2m-1)^{-\delta_i}$.
	\end{proof}
	Following Ollivier, we find that the probability that a $\lfloor(2m-1)^{3d}\rfloor$--tuple $R$ of random triangular relations contains an injective $n$--sub-tuple that fulfils $Y$ satisfies
	\[
	\leq \exp\left(\log(2m-1)\cdot \frac{1}{2}\left(\frac{3|Y|+2(\Red(Y)-\Cancel(Y))}{|Y|} - 3(1-2d)\right)\right),
	\]
	which by our assumption on $Y$ is $< \exp\left(\log(2m-1)\cdot -3\epsilon \right)$,
	which goes to $0$ as $m\to\infty$, whence the result follows.
\end{proof}

\section{Colliding geodesics}

We show that geodesics tend to collide in a random triangular group at reasonably low densities.
Our key proposition is the following.

\begin{proposition}\label{prop:key2}
	Let $d_0 \in (0,\bestboundx)$. Then there exist $\delta,L,k$ depending on $d_0$ so that for $d\leq d_0$, a random triangular group presentation $G$ at density $d$ with Cayley graph $X$ has a.a.s.\ that $X$ is $\delta$-hyperbolic, and for every $x,y \in X$ with $d(x,y) \geq L$, for every geodesic $\gamma$ from $x$ to $y$, there exist at most $k$ elements $g_1,\dots,g_l$ of $G_{x,y,800\delta}$  such that for every $1\leq i<j\leq l$ we have $g_i\gamma\cap g_j\gamma \cap B(x,L+800\delta) = \emptyset$.
\end{proposition}
\begin{remark}
  Note that we allow $g_i=1_G$. 
  The proposition implies 
  that while $k$ parallel geodesics may be 
  possible, any additional geodesic would have to collide into one of the first $k$.
  Indeed, at densities $d>\frac{1}{3}$ we likely have a relation of the form $ab^2$, and so $\gamma=(a^i)$, $b^{-1}\gamma$ are parallel, i.e.\ $k\geq 2$, see Figure~\ref{fig:abb}.
\end{remark}
\begin{figure}
   \def\svgwidth{.5\textwidth}
   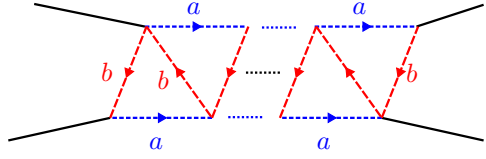
   \caption{Parallel geodesics}
   \label{fig:abb}
\end{figure}

\begin{proof}[Proof of Theorem~\ref{thm:acylindrical}]
	Fix $d_0, d \leq d_0, \delta=\delta(d_0), k=k(d_0)$, and $L=L(d_0)$ from Proposition~\ref{prop:key2}. 
	Suppose $x,y \in X$ are given with $d(x,y) \geq L$.
	Let $\gamma:[0,d(x,y)]\ra X$ be any geodesic from $x$ to $y$, and 
	let $\{g_1,g_2,\dots, g_l\}$ be a maximal subset of $G_{x,y,800\delta}$ such that for every $1\leq i<j\leq l$ we have $g_i\gamma\cap g_j\gamma \cap B(x,L+800\delta)=\emptyset$. Then, by Proposition~\ref{prop:key2}, we have $l\leq k$.
	Given any $h \in G_{x,y,800\delta}$, by maximality of $\{g_1,g_2,\dots, g_l\}$ we have for some $1\leq i\leq l$ that  
	$g_i\gamma \cap h\gamma \cap B(x,L+800\delta)$ is non-empty.
	Thus there exists $t,t'\leq L+2\cdot 800\delta$ so that 
	$g_i\gamma(t)=h\gamma(t')$,
	hence by the freeness of $G$ acting on $X$ we must have $h = g_i\gamma(t)\gamma(t')^{-1}$.
	Thus there are at most $N=k(L+2\cdot800\delta)^2$ possibilities for $h$.
\end{proof}

\begin{proof}[Proof of Proposition~\ref{prop:key2}]
	Let $d_0<\bestboundshort$, $d\leq d_0$, and $d':=(d_0+\bestboundshort)/2$. Notice that $d'=d+\epsilon$ for some $\epsilon>0$, and $d'<\bestboundshort$.
	Set $\delta:=12/(1-2d_0)$. Let $k,L,M$ be (large) integers to be determined that will depend only on $d_0$, with $k\geq 1$ and $L\geq 4\cdot800\delta+4\delta+2$.
	
	Since the conclusions of Theorem~\ref{thm:isop2} and Theorem~\ref{thm:isop-reductions2} (for our fixed values of $d$, $\epsilon$, and $M$) hold a.a.s.\ in a triangular random group, we from now on assume we have a group $G$ given by a triangular presentation for which these conclusions hold. The claim on $\delta$-hyperbolicity is immediate from Theorem~\ref{thm:isop2}.

	Suppose we are given $x,y \in X$ with $d(x,y) \geq L$.  Fix $\gamma$ a geodesic joining $x$ to $y$. Recall that $G_{x,y,800\delta}=\{g\in G:d(x,gx)\leq 800\delta,d(y,gy)\leq 800\delta\}$.

	\medskip

	{\noindent \bf Notation for and dependencies of constants.} In the following, for $i\in\N$ each $C_i$ denotes an appropriate constant $C_i = C_i(\lambda_1,\dots) \in \R$ depending only on $\lambda_1,\dots$ and chosen appropriately to make the inequalities work. 
	Observe here that $\delta$ and $d'$ are functions of $d_0$, whence in the following we will be able to write $C_i(d_0)$ instead of $C_i(\delta)$, respectively $C_i(d')$.
	
	At the end of the proof, we will be able to express $k$ as a function of $d_0$ and $L$ as a function of $d_0$ and $k$, thus making all constants we produce from $d_0,k,L$ depend only on $d_0$.

	\medskip
	\noindent{\bf A diagram $D$ from two non-colliding geodesics.}
	Suppose $g_1,g_2 \in G_{x,y,800\delta}$ are such that $g_1\gamma \cap g_2\gamma \cap B(x,L+800\delta) = \emptyset$. 
	Consider a geodesic joining $g_1x$ to $g_2x$, and take a subpath $\alpha_1$ meeting $g_1\gamma$ and $g_2\gamma$ exactly once each.
	The first vertex in $g_1\gamma$ at distance at least $L-2\cdot 800\delta-2\delta-1$ from $g_1x$ is at distance at most $2\delta$ from some vertex in $g_2\gamma$ by hyperbolicity and the choice of $L$; consider a geodesic between these two vertices and choose a subpath $\alpha_2$ meeting each $g_1\gamma$ and $g_2\gamma$ in exactly one point.
	Together, $\alpha_1, \alpha_2$ and the appropriate subpaths of $g_1\gamma$, $g_2\gamma$ make an embedded loop in $X$ contained in $B(x,L+800\delta)$. 
	The $\alpha_1,\alpha_2$ sides have lengths totalling $\leq 2\cdot 800\delta+2\delta$ and the sides in $g_1\gamma, g_2\gamma$ have lengths $\in [L-4\cdot800\delta-4\delta-1, L]$, moreover, they contain identically labelled subpaths of length $\geq L-4\cdot800\delta-4\delta-1$.

	Let $D$ be a minimal area van Kampen diagram for this loop, necessarily homeomorphic to a disc.

	\medskip
	
	\noindent{\bf Three types of edges in $\partial D$.}
	Since $g_1\gamma, g_2\gamma$ are disjoint geodesics and the presentation is triangular, we can partition the edges in $\partial D$ into three categories:
	\begin{enumerate}
	  \item[0.] it comes from $\alpha_1$ or $\alpha_2$;
	  \item[1.] it comes from $g_1\gamma$ or $g_2\gamma$, and the remaining vertex in the $2$--cell containing it is not in $g_1\gamma\cup g_2\gamma$;
	  \item[2.] it comes from $g_1\gamma$ or $g_2\gamma$, and the remaining vertex in the $2$--cell containing it is in $g_1\gamma\cup g_2\gamma$.
	\end{enumerate}
	Notice that in case 2, if the edge $e$ is in $g_1\gamma$, then the remaining vertex must be in $g_2\gamma$ because $\gamma$ is geodesic, and vice versa if $e$ is in $g_2\gamma$.
	Let us write $\partial D = E_0(D) \sqcup E_1(D) \sqcup E_2(D)$ for this partition; if $D$ is clear, we write $E_0=E_0(D)$, and so on.
	Thus $|\partial D| = |E_0|+|E_1|+|E_2|$.  Note that $|E_0| \leq 2\cdot 800\delta+2\delta$.

	\begin{lemma}\label{lem:e1} There exists a constant $A_1(d_0)$ depending only on $d_0$ such that $$|E_1|\leq \frac{4(3d'-1)}{3(1-2d')}\cdot 2L+A_1(d_0).$$	 
	\end{lemma}

	In other words, when $d$ is only a little more than $1/3$ and $L \geq (|E_1|+|E_2|)/2+C_1(d_0)$ is large, the fraction $|E_1|/(|E_1|+|E_2|)$ is small.  So most edges are in $E_2$, and the corresponding triangles meet the other side of $D$.

	\medskip

	\begin{proof}	
	Let $V$ be the number of vertices in $D$.
	Since $\gamma$ is a geodesic, for each edge in $g_1\gamma \cap E_1$, the remaining vertex in the $2$--cell containing it can correspond to at most two such edges in $g_1\gamma \cap E_1$.  Taking into account $g_2\gamma$ as well and the fact that $V-|E_1|-|E_2|$ is an upper bound for the number of vertices not in $g_1\gamma\cup g_2\gamma$, we obtain $|E_1| \leq 4(V-|E_1|-|E_2|)$.

	\medskip

	The total number of edges in $D$ is $\frac{1}{2}( 3|D|+|\partial D|)$, so by Euler's formula $V-E+F=1$, we have
	$V = 1+ \frac{1}{2}( 3|D|+|\partial D| )-|D|=1+\frac{1}{2}(|D|+|\partial D|)$.

	This, along with Theorem~\ref{thm:isop2}, gives
	\begin{align*}
	  |E_1| & \leq  4\left(V - (|E_1|+|E_2|) \right)
	   = 4 \left( 1+ \frac12\left( |D| + |E_0| - (|E_1|+|E_2|) \right)\right)
	   \\ & \leq 2|D| - 2(|E_1|+|E_2|) + C_2(d_0)
	   \\ & \leq 2 \frac{|E_1|+|E_2|}{3(1-2d')}-2(|E_1|+|E_2|) + C_3(d_0)
	   = \frac{4(3d'-1)}{3(1-2d')}(|E_1|+|E_2|) + C_3(d_0)
	   \\ & \leq \frac{4(3d'-1)}{3(1-2d')}\cdot 2L + C_4(d_0)
	\end{align*}
	We set $A_1(d_0):=C_4(d_0)$.
	\end{proof}

	\medskip
	\noindent{\bf A complex $Y$ from $k+1$ non-colliding geodesics.}
	Now, in order to produce a contradiction in the end, suppose we have $k+1$ elements $\{g_1,\ldots,g_{k+1}\} \subset G_{x,y,800\delta}$ so that the geodesics $g_i\gamma$ are pairwise disjoint in $B(x,L+800\delta)$.
	For each pair $(g_i\gamma,g_j\gamma)$, $i < j$, we can do the construction above to find a reduced van Kampen diagram $D_{i,j}$.  For each $i$ there is a common subpath of the geodesic $g_i\gamma$ of length $\geq L-4\cdot800\delta-4\delta-1$ which appears in the boundary of each diagram $D_{i,j}$ for $j>i$ and $D_{j,i}$ for $j<i$.
	We construct a labelled 2-complex $Y\to K$ from these $\binom{k+1}{2}$ diagrams by identifying, for each $i$, the paths corresponding to $g_i\gamma$ in the $k$ diagrams containing such a path. (For the purpose of illustration, observe that $Y$ is homeomorphic to the product of a complete graph on $k+1$ vertices with a compact interval.)   There exists an upper bound depending only on $d_0,k,L$ for the number of $2$--cells $Y$, and we choose $M$ to be that bound. Therefore, we can apply Theorem~\ref{thm:isop-reductions2} to $Y$.

	\begin{lemma}[Lower bound for $\Red(Y)$]\label{lem:lowerred} There exists $A_2(d_0,k)$ depending only on $d_0$ and $k$ such that
	\begin{align*}
		\Red(Y) \geq \frac{3}{2}(1-2d')\binom{k+1}{2}2L + \frac{(k+1)(k-2)}{2}L +A_2(d_0,k),
	\end{align*}	
	\end{lemma}

	\begin{proof}
		To estimate $\Red(Y)$ from below, by Theorem~\ref{thm:isop-reductions2} is suffices to estimate $\Cancel(Y)$ from below.
	In each van Kampen diagram $D_{i,j}$, 
	$3|D_{i,j}|=|\partial D_{i,j}|+2\Cancel(D_{i,j})\leq 2L+ C_5(d_0)+2\Cancel(D_{i,j})$.
		So the contribution to the sum in $\Cancel(Y)=\sum_{e\in Y^{(1)}} (\deg(e)-1)_+$ coming from those $e$ which are interior edges of some $D_{i,j} \subset Y$ for some $i<j$ is at least
	\[
	\sum_{i<j} \left(\frac32 |D_{i,j}| -L-\frac{C_5(d_0)}{2}\right) \geq \frac32 |Y| - \binom{k+1}{2}L + C_6(d_0,k).
	\]
		Meanwhile, for each of the at least $(k+1)L+C_7(d_0,k)$ edges glued together to make $Y$ there is a contribution of $k-1$ to the sum in $\Cancel(Y)$.
	So in total,
	\begin{align*}
	\Cancel(Y) &\geq \frac32|Y| - \binom{k+1}{2}L+(k+1)(k-1)L+C_8(d_0,k) \\
		&= \frac32|Y| +\frac{(k+1)(k-2)}{2}L+C_8(d_0,k).
	\end{align*}
	So by Theorem~\ref{thm:isop-reductions2} we have
	\[
		\frac32|Y| +\frac{(k+1)(k-2)}{2}L+C_8(d_0,k)
		\leq d'|Y|3 +\Red(Y),
	\]
	i.e.\
	\begin{align*}
		\Red(Y) & \geq \frac32(1-2d')|Y| +\frac{(k+1)(k-2)}{2}L+C_8(d_0,k)
		\\ & \geq \frac{3}{2}(1-2d')\binom{k+1}{2}2L + \frac{(k+1)(k-2)}{2}L +C_{9}(d_0,k),
	\end{align*}
	where the second inequality follows from $|D_{i,j}| \geq 2L+C_{10}(d_0)$ since each $2$--cell can meet at most one edge in the disjoint bounding geodesics. We set $A_2(d_0,k):=C_{9}(d_0,k)$.
	\end{proof}
	
	\medskip
	
	\begin{lemma}[Upper bound for $\Red(Y)$]\label{lem:upperred} Let $A_1(d_0)$ be as in Lemma~\ref{lem:e1}. Then 
	\[
		\Red(Y) \leq\binom{k+1}{2} \frac{4(3d'-1)}{3(1-2d')} \cdot 2L +\binom{k+1}{2}A_1(d_0).
	\]
	 
	\end{lemma}
	
	\begin{proof}
		Recall $\Red(Y) = \sum_{e \in Y^{(1)}}\sum_{i \in I} \left(\sum_{f\in Y^{(2)}, \pi(f)=i} \xi(e,f)-1\right)_+$.
		Since each $D_{i,j} \subset Y$ is a reduced van Kampen diagram, the summand corresponding to any interior edge $e$ of $D_{i,j}$ contributes $0$ to $\Red(Y)$.
		Thus all contributions to $\Red(Y)$ come from edges in the $k+1$ glued geodesics.

		If for some $e \in Y^{(1)}$ and $i \in I$ we have a summand $1\leq t=(\sum_{f\in Y^{(2)},\pi(f)=i}\xi(e,f)-1)_+$ in $\Red(Y)$, then at least $t$ of the $\geq t+1$ $2$--cells with $\pi(f)=i$ in which $e$ occurs in the same (minimal) position must have $e$ belonging to $E_1$ in their diagrams. Otherwise, consider two of these $2$--cells $f_2$ and $f_3$ which have $e$ in $E_2$ for their diagrams, say for example $f_2$ is in $D_{1,2}$ and $f_3$ is in $D_{1,3}$, and $e \in E_2(D_{1,2})$, $e\in E_2(D_{1,3})$, and $\xi(e,f_2)=\xi(e,f_3)=1$. The 1-skeleton of the van Kampen diagram $D_{1,2}\cup_{g_1\gamma}D_{1,3}$ admits a well-defined map $f$ to $X$, and has a possible reduction across $e$. If, for $i= 2,3$, we denote by $v_i$ the vertices of $f_i$ not in $e$, then $f(v_2)=f(v_3)$, i.e.\ $g_2\gamma\cap g_3\gamma\cap B(x,L+800\delta)\neq\emptyset$, contradicting our assumption. Thus, using Lemma~\ref{lem:e1}, we have
	\[
		\Red(Y) \leq \sum_{i<j} |E_1(D_{i,j})| \leq\binom{k+1}{2} \frac{4(3d'-1)}{3(1-2d')} \cdot 2L +\binom{k+1}{2}A_1(d_0). \qedhere
	\]
	\end{proof}

	\noindent{\bf Conclusion.}
	Let $A_3(d_0,k):=\binom{k+1}{2}A_1(d_0)-A_2(d_0,k)$. We show that whenever $k$ is large enough, depending on $d_0$ and, subsequently, $L$ is large enough, depending on $d_0$ and $k$, we have
	\[
	 \binom{k+1}{2}\cdot \frac{4(3d'-1)}{3(1-2d')}\cdot 2 +\frac{A_3(d_0,k)}{L}
	 < 
	 \frac{3}{2}(1-2d')\binom{k+1}{2}2 + \frac{(k+1)(k-2)}{2},
	\]
	which, using Lemmas \ref{lem:lowerred} and \ref{lem:upperred} implies the contradiction $\Red(Y)<\Red(Y)$. The above is equivalent, via dividing through by $(k+1)k$, to 
	\[
	 \frac{4(3d'-1)}{3(1-2d')} +\frac{A_3(d_0,k)}{L (k+1)k}
	 < \frac{3}{2}(1-2d') + \frac{(k+1)(k-2)}{2(k+1)k} 
	 = (2-3d') - \frac{1}{k} .
	\]
	The condition
	\[
	 \frac{4(3d'-1)}{3(1-2d')} < 2-3d'
	\]
	is equivalent to $d' < \bestboundshort$, which we assumed. Thus, we may choose $k$ only depending on $d'$, which in turn only depends on $d_0$, such that
	\[
	 \frac{4(3d'-1)}{3(1-2d')} < 2-3d'-\frac{1}{k}.
	\]
	Subsequently we choose $L$ such that $\Red(Y)<\Red(Y)$, contradicting our assumptions.
\end{proof}

\bibliographystyle{amsalpha} %
\bibliography{biblio} %

\end{document}